\documentclass[11pt]{amsart}
\usepackage{amsmath}
\usepackage{amsfonts}
\usepackage{amssymb}
\usepackage{mathrsfs}
\usepackage{amsthm}
\usepackage{graphicx}
\usepackage[svgnames,dvipsnames]{xcolor}
\usepackage{float}
\usepackage{diagbox}  
\usepackage[textwidth=22mm]{todonotes}
\usepackage[normalem]{ulem}
\usepackage{hyperref}

\hypersetup{
    linktocpage=true,
    colorlinks,
    linkcolor={BrickRed},
    citecolor={ForestGreen!90!black},
    urlcolor={Salmon!50!black}
}

\numberwithin{equation}{section}

\newtheorem{theorem}{Theorem}
\newtheorem{lemma}{Lemma}
\newtheorem{corollary}{Corollary}
\newtheorem{conjecture}{Conjecture}

\theoremstyle{definition}
\newtheorem{remark}{Remark}

\DeclareMathOperator{\Tr}{Tr}

\newcommand{\ad}{\mathrm{ad}}

\newcommand{\adstar}{\ad^{\star}}

\newcommand{\p}{\textcolor{red}{P}}
\newcommand{\n}{\textcolor{blue}{N}}

\DeclareMathOperator{\Ric}{Ric}

\newcommand\sixJ[3]{\left\{\begin{matrix}
#1 & #2 & #3 \\
\frac{N-1}{2} & \frac{N-1}{2} & \frac{N-1}{2}
\end{matrix}
\right\}}
\newcommand{\weight}{\alpha}
\newcommand{\soop}{\mathcal{C}}

\newcommand{\Diff}{\mathcal{D}}

\title{Ricci curvature for hydrodynamics on the sphere}

\author{Leandro Lichtenfelz}
\author{Klas Modin}
\author{Stephen C. Preston}

\address{L. Lichtenfelz: Wake Forest University, Winston-Salem, NC., U.S.A.}
\email{lichtel@wfu.edu}
\address{K. Modin: Mathematical Sciences, Chalmers and University of Gothenburg, Gothenburg, SE-412~96, Sweden}
\email{klas.modin@chalmers.se}
\address{S. C. Preston: Department of Mathematics, CUNY Brooklyn College, New York, NY 10468 
and CUNY Graduate Center, New York, NY 10016, USA}
\email{Stephen.Preston@brooklyn.cuny.edu}

\begin{document}

\begin{abstract}
    The geometric description of incompressible hydrodynamics, as geodesic motion on the infinite-dimensional group of volume-preserving diffeomorphisms, enables notions of curvature in the study of fluids in order to study stability.
    Formulas for Ricci curvature are often simpler than those for sectional curvature, which typically takes both signs, but the drawback is that Ricci curvature is rarely well-defined in infinite-dimensional spaces.
    Here we suggest a definition of Ricci curvature in the case of two-dimensional hydrodynamics,  
    based on the finite-dimensional Zeitlin models arising in quantization theory, which gives a natural tool for renormalization.
    We provide formulae for the finite-dimensional approximations and give strong numerical evidence that these converge in the infinite-dimensional limit, based in part on four new conjectured identities for Wigner $6j$ symbols.
    The suggested limiting expression for (average) Ricci curvature is surprisingly simple and demonstrates an average instability for high-frequency modes which helps explain long-term numerical observations of spherical hydrodynamics due to mixing.
    \\[2ex]
    \\[2ex]
\end{abstract}

\maketitle 

\tableofcontents

\section{Introduction}

The incompressible Euler equations describe the motion of an ideal, inviscid fluid on a Riemannian manifold $(M, g)$. 
If $u$ denotes the time-dependent vector field on $M$ representing the velocity of the fluid, the Euler equations are
\begin{gather}\label{eq_euler}
\begin{split}
\partial_tu + \nabla_uu &= -\nabla p, \\
\mathrm{div}\, u &= 0, \\
u(0) &= u_0
\end{split}
\end{gather}
where $\nabla_uu$ is the covariant derivative of $u$ along itself and $p$ is the pressure function. 

Arnold~\cite{arnold2013differential} reinterpreted the system~\eqref{eq_euler} as geodesic equations on the infinite-dimensional Lie group $\mathrm{Diff}_{\mu}(M)$ of volume-preserving diffeomorphisms of $M$, equipped with the right-invariant $L^2$ Riemannian metric. 
Through this discovery, geometric concepts lend themselves to hydrodynamical interpretations.
In particular, since fluid trajectories are geodesics, positive curvature leads to convergence of nearby flows, indicating stability, while negative curvature has the opposite effect, leading to instability. 
Motivated by this perspective, Arnold computed sectional curvatures of $\mathrm{Diff}_{\mu}(M)$ when $M$ is the two-torus, and subsequent work extended these computations to various settings, including both two- and three-dimensional manifolds $M$ (see, e.g.,  \cite{lukatskii1984curvature,rouchon1992thejacobi,misiolek1993stability,preston2005nonpositive,smolentsev2006diffeomorphism}). Collectively, these results show that sectional curvatures of $\mathrm{Diff}_{\mu}(M)$ often, but not always, turn out negative.

Specifically for $\mathrm{Diff}_{\mu}(S^2)$, which is the most important for studying large-scale weather, curvature computations have been performed by many authors in terms of spherical harmonics (see Suri~\cite{Suri} and references therein for a recent explication). Rouchon~\cite{rouchon1992thejacobi} showed that curvatures in sections containing at least one rotation field are always nonnegative, but that for any other velocity field there is a section containing it with negative curvature. The third author showed~\cite{preston2005nonpositive} the opposite, that for every divergence-free velocity field on $S^2$, there is always a section containing it with positive curvature; this is false for most other surfaces. Lukatskii~\cite{lukatskii1988thestructure} considered the curvatures in sections containing $u=\nabla^{\perp}f$ for a function $f(z)$ on $S^2$ and found that in the basis $v_m$ of spherical harmonics, $K(u, v_m) \to -f''(0)^2/\lVert u\rVert^2_{L^2}$, which is the closest result to the present study in the literature and suggests the curvatures are ``mostly negative'' when considered in this basis. It is not at all obvious whether such a result holds in a different basis. On the other hand Suri~\cite{suri2024conjugate} showed that almost every spherical harmonic flow has conjugate points along it, indicating existence of many sections of positive curvature, often found in off-diagonal directions. Hence quantifying the amount of positive or negative curvature is far from clear in this infinite-dimensional setting, particularly in a basis-independent way. Ricci curvature makes this precise in an invariant way but ordinarily only gives finite values in finite-dimensional manifolds\footnote{As a simple example, one can take the sphere of radius $r$ in $L^2(M)$, which has sectional curvature $1/r$ in any plane. Thus, Ricci clearly diverges since there are infinitely many orthogonal planes.}.

As further motivation, in the special case of two-dimensional ideal hydrodynamics, the scalar vorticity field is advected by the geodesic path of diffeomorphisms.
Onsager~\cite{On1949} predicted that regions of equally signed vorticity tend to merge to form vortex condensates.
However, since vorticity is advected, the merging is not via diffusion, but via intricate thinning and folding called \emph{mixing}.
This way, condensates of vorticity develop on large scales at the cost of increasing entanglement on small scales.
The long-time behavior of this dynamics comprises \emph{two-dimensional turbulence} (cf.~Bofetta and Ecke~\cite{BoEc2012}).
But we still lack an ``Arnold-like'' geometric understanding of mixing and vortex condensation, and here Ricci curvature may offer insights.
Indeed, mixing can be interpreted as sensitivity relative to initial data, which in turn is connected to the growth of Jacobi fields. 
More precisely, Ricci curvature governs the evolution of infinitesimal volumes spanned by Jacobi fields, in such a way that negative values accelerate their growth.
In this sense, average Ricci curvature is related to average mixing. Heuristically, we expect that mixing arises from exponential growth of Jacobi fields, driven by negative curvature, with a Lyapunov exponent proportional to the square root of the negative Ricci curvature.

In summary, there are compelling reasons to seek a well-defined Ricci tensor for $\mathrm{Diff}_\mu(M)$, but in order to do this rigorously it is necessary to view it as convergence of Ricci tensors of finite-dimensional geometric approximations.
In this paper, we introduce a notion of Ricci curvature, through quantization, for $\mathrm{Diff}_{\mu}(M)$ when $M$ is the two-sphere $S^2$. Numerical computations suggest that this quantity, defined as a renormalized series over Ricci curvatures in the quantized models, not only converges to a negative value, in line with expectations from the aforementioned sectional curvature computations, but is also given by a remarkably simple formula (Conjecture \ref{conjecture_ricci}). 
Motivated by the statistical long-time behaviour of hydrodynamics on the sphere, resulting in motion that on the large scale is nearly steady up to rigid rotations\footnote{These results are obtained in both numerical simulations~\cite{DrQiMa2015,MoVi2020} and physical experiments~\cite{Co1984}.}, we also compute the Ricci curvature for a quantization model of the homogeneous space $\mathrm{Diff}_\mu(S^2)/\mathrm{SO}(3)$. This homogeneous space is important due to the observation that long-term dynamics of 2D Euler appears to converge to a system on $SO(3)$ with high-frequency terms orthogonal to it fading into noise, suggesting a uniformity of the behavior in the high-frequency quotient space.

The starting point for our definition is the Zeitlin model~\cite{Ze1991}, a sequence of finite-dimensional approximations of the group  $\mathrm{Diff}_{\mu}(S^2)$ of area-preserving diffeomorphisms of the sphere based on a quantization scheme introduced by Hoppe~\cite{Ho1989}. In this model, the special unitary groups $SU(N)$ serve as finite-dimensional analogues of $\mathrm{Diff}_{\mu}(S^2)$ when equipped with a particular right-invariant metric, known as the Zeitlin metric. These groups capture the structure of, and converge to, the area-preserving diffeomorphism group in the large-$N$ limit; for the precise convergence statement for Zeitlin's model on the sphere, see \cite{MoVi2024}.

Our main results begin with a structural theorem: the Ricci tensor of the Zeitlin metric in any dimension is block diagonal with respect to a natural decomposition $\{V_{\ell}\}_{\ell=1}^{N-1}$ of $\mathfrak{su}(N)$ into irreducible $\mathfrak{so}(3)$ modules, acting as a scalar multiple $r_\ell(N)$ of the identity on each subspace $V_\ell$. We derive exact formulas for the values $r_{\ell}(N)$ 
and also determine the Ricci curvature of the homogeneous space $SU(N)/SO(3)$, which corresponds to fluid motion modulo rigid rotations. Finally, we analyze the asymptotic behavior of $r_{\ell}(N)$ as $N \to \infty$. This asymptotic analysis leads naturally to our definition of averaged Ricci curvature for the infinite-dimensional limit. 
In addition it leads to four new conjectured identities for Wigner $6j$ symbols for which there is strong numerical evidence.

We point out that the block Ricci result does not follow from the general theory of ``nice'' bases for Lie groups (cf.~\cite{krishnan2021diagonal}), but relies on the specific form of the Zeitlin metric, which is dictated by the quantized Laplacian. In particular, one can check directly that there exist right-invariant metrics on $SU(N)$ near the Zeitlin metric which do not possess this property. On the other hand, general properties of irreducible representations show that one should expect that the Ricci curvature, as a quadratic form invariant under the isometric action of $SO(3)$, should also be a multiple of the metric on each subspace $V_{\ell}$ by Schur's Lemma; see for example Park-Sakane~\cite{park1997invariant} for a sketch of the argument. Here we verify this statement and more importantly compute the multiples.

Concerning applications in two-dimensional hydrodynamics, our result on positivity of Ricci curvature (see Corollary~\ref{positive_v1_corollary} below) in the lowest wavenumber eigenspace $V_1$ implies that increased angular momentum on average has a stabilizing effect on the dynamics. By contrast, our conjecture that the Ricci curvature is eventually negative (see Conjecture~\ref{conjecture_ricci} below) in all higher eigenspaces hints that on average the sectional curvature is indeed negative, as previously suggested in the literature. 
As a whole, the results are compatible with mixing observed numerically, where the large scales typically settle on a near global rotation (where Ricci is positive), whereas small scales (where Ricci is conjectured negative) are fully mixed \cite{MoVi2020}.

It is possible to carry out this construction for other closed surfaces, such as the torus. An earlier approach, not based on quantization, was proposed by Lukatskii \cite{lukatskii1984curvature} for the flat torus $\mathbb{T}^n$, but without using quantized geometry and hence innately basis-dependent. 
An open problem is thus to derive Ricci curvature as in this paper but for the flat torus  using Zeitlin's quantized model for $\Diff_{\mu}(\mathbb{T}^2)$. 
Another is to extend the results to the axisymmetric 3D Euler equations on $S^3$, which can also be addressed via quantization \cite{MoPr2025}. Ricci curvature was computed there only in the simplest case $N=2$, but for higher $N$ one may expect similar formulas as found here.

This paper is organized as follows. Section \ref{main_results} contains an overview of our  results, including some background.
Section \ref{sec_ricci_finite_N} contains the computation of the Ricci tensor and the proof of Theorem \ref{main_thm}. In Section \ref{section_asymptotic}, we investigate the limiting behavior of Ricci curvature for large $N$, making use of some new conjectured identities for Wigner $6j$ symbols. Finally, in Section \ref{homogeneous_section} we determine the Ricci curvature of the homogeneous space $SU(N)/SO(3)$.

\textbf{Acknowledgements:}
This work was supported by the Swedish Research Council (grant number 2022-03453), the Knut and Alice Wallenberg Foundation (grant numbers WAF2019.0201), and the Göran Gustafsson Foundation for Research in Natural Sciences and Medicine.
The computations were enabled through resources provided by Chalmers e-Commons at Chalmers and by the National Academic Infrastructure for Supercomputing in Sweden (NAISS), partially funded by the Swedish Research Council through grant no.~2022-06725. The first author was supported by the A.J. Sterge Faculty Fellowship. We also thank Akshay Venkatesh for helpful correspondence.

\section{Main results}\label{main_results}

In this section, we state our main results, including only the minimal background needed for the formulation of the theorems. 
A more detailed discussion of the quantization procedure and the structure of the Zeitlin model is deferred to the next section.

\subsection{Block-Einstein structure}

The Lie algebra $\mathfrak{su}(N)$ admits a decomposition
\begin{equation}\label{eq_su_n_main}
\mathfrak{su}(N) = V_1 \oplus \cdots \oplus V_{N-1},
\end{equation}
where each $V_\ell$ is an irreducible representation of $\mathfrak{so}(3)$, has dimension $2\ell + 1$ and carries a distinguished basis 
$\{T_{\ell m} : -\ell \leq m \leq \ell\}$, which can be thought of as a finite-dimensional truncation of the classical spherical harmonics on $S^2$ (cf. \cite{MoVi2024}, \cite{Ze2004}).

The quantized Laplacian $\Delta_N : \mathfrak{su}(N) \rightarrow \mathfrak{su}(N)$ is defined so as to preserve this decomposition and mimic the action of the standard Laplace-Beltrami operator on the two-sphere, assigning eigenvalue $-\ell(\ell+1)$ to each $V_\ell$. Concretely we take a representation of $\mathfrak{so}(3)$ by operators $\soop_i\colon \mathfrak{su}(N)\to\mathfrak{su}(N)$
satisfying $[\soop_i,\soop_j]=\epsilon_{ijk}\soop_k$, and define the Laplacian by $\Delta_N = \sum_i \soop_i^2$, verifying that it commutes with each $\soop_i$. We then construct the special basis $T_{\ell m}$ satisfying $\Delta_N T_{\ell m} = -\ell(\ell+1)T_{\ell m}$ and $\soop_3(T_{\ell m}) = mT_{\ell,-m}$, which uniquely determines the vectors up to scaling for each $m\in \{-\ell,\ldots, \ell\}$. We consider two Riemannian metrics on $SU(N)$: the bi-invariant metric and the Zeitlin metric, the latter being only right-invariant. They are given at the tangent space to the identity by 
\begin{align}
(u, v) &= \mathrm{Tr}\big(u^\dagger v\big), \qquad &\text{(bi-invariant metric)} \label{eq_biinvariant_metric}
\\
\langle u, v\rangle &= \frac{1}{N} \operatorname{Tr}\big(u^\dagger (-\Delta_N) v\big), &\text{(Zeitlin metric)} \label{eq_zeitlin_metric}
\end{align}
for $u, v \in \mathfrak{su}(N)$.
The basis $\{ T_{\ell m} \}$ is scaled so it is orthonormal with respect to the bi-invariant metric.
Furthermore, for the Lie algebra structure, the usual matrix commutator is denoted $[\cdot,\cdot]$.
However, in quantization theory one needs to work with the scaled bracket $\frac{1}{\hbar_N}[\cdot,\cdot]$ where $\hbar_N = 2/\sqrt{N^2-1}$.
The infinitesimal (right) adjoint action is therefore defined with respect to this scaled Lie bracket, so that $$\ad_uv = -\frac{1}{\hbar_N}[u, v].$$

The relevance of Zeitlin's metric lies in the fact that its geodesics provide a finite-dimensional approximation of hydrodynamics on the sphere, as explained later in the next section.
Our first theorem says that the Ricci curvature of Zeitlin's metric~\eqref{eq_zeitlin_metric} is a multiple of the identity when restricted to each subspace $V_{\ell}$, which means that the Zeitlin model is ``block-Einstein.''

\begin{theorem}\label{main_thm}
The Ricci curvature of $SU(N)$ under the Zeitlin metric \eqref{eq_zeitlin_metric} is block diagonal with respect to the decomposition \eqref{eq_su_n_main}. 
That is, the Ricci tensor $\Ric$ acts as a scalar multiple of the metric on each subspace $V_{\ell}$ for $\ell<N$:
for $u\in V_{\ell}$ and $v\in V_{\ell'}$ we have 
\begin{equation*}
    \Ric(u,v) =  
    r_{\ell}(N) \, \langle u,v\rangle. 
\end{equation*}
Furthermore, we have $r_{\ell}(N) = r_{\ell}(N)^+ - r_{\ell}(N)^-$, where $r_{\ell}(N)^{\pm}$ are both nonnegative and given explicitly by 
\begin{gather}\label{eq_ricci_intro}
\begin{split}
r_{\ell}(N)^+ &:= \frac{N}{\hbar_N^2} \sum\limits_{\substack{k, k' = 1 \\ k+k'+\ell\;\mathrm{odd}}}^{N-1} \frac{\lambda_{\ell}}{\lambda_{k}\lambda_{k'} } 
\, (2k+1)(2k'+1)  \left\{\begin{matrix}
\ell & k & k' \\
\frac{N-1}{2} & \frac{N-1}{2} & \frac{N-1}{2}
\end{matrix}
\right\}^2,
\\ 
r_{\ell}(N)^- &:= \frac{N}{\hbar_N^2} \sum\limits_{\substack{k, k' = 1 \\ k+k'+\ell \;\mathrm{odd}}}^{N-1} \frac{  (\lambda_{k}-\lambda_{k'})^2}{\lambda_{k}\lambda_{k'}\lambda_{\ell}} 
\, (2k+1)(2k'+1)  \left\{\begin{matrix}
\ell & k & k' \\
\frac{N-1}{2} & \frac{N-1}{2} & \frac{N-1}{2}
\end{matrix}
\right\}^2.
\end{split}
\end{gather}
Here, $\lambda_i = i(i+1)$, $\hbar_N = 2/\sqrt{N^2-1}$ and $\{:::\}$ denotes the Wigner $6j$ symbol. 
\end{theorem}

For a detailed account of Wigner symbols, see \cite{Varshalovich1988quantum}. The quotient space $SU(N)/SO(3)$ has a natural interpretation as the configuration space of fluid motion modulo rigid rotations. We also determine the Ricci curvature of this quotient space explicitly, which turns out to be strictly larger than the Ricci curvature of $SU(N)$ by an amount that depends on the wavenumber, but not the dimension $N$. Since we expect $r_{\ell}(N)$ to grow like $N^2$ for fixed $\ell$ (see Theorem \ref{theorem_3} below), the Ricci curvature of the quotient behaves in essentially the same way as in $SU(N)$ for large $N$.

\begin{theorem}\label{homogeneous_space_thm}
The Ricci tensor $\Ric_B$ of the quotient space $B = SU(N)/SO(3)$, where $SU(N)$ is equipped with the Zeitlin metric \eqref{eq_zeitlin_metric}, acts as a scalar multiple of the metric on each subspace $V_{\ell}$ for $2 \leq \ell \leq N-1$: for $u \in V_{\ell}$ and $v \in V_{\ell'}$ we have 
\begin{gather}
\begin{split}
\Ric_B(u, v) =  
\left(r_{\ell}(N) + \frac{3}{\lambda_{\ell}}\right)\langle u, v \rangle_B
\end{split}
\end{gather}
where $\langle \,,\,\rangle_B$ is the quotient metric on $B$, and $r_{\ell}(N)$ is as in Theorem \ref{main_thm}.
\end{theorem}

Positivity of Ricci curvature in some special directions, and all dimensions, will follow easily from the formulas in Theorem \ref{main_thm}. 

\begin{corollary}\label{positive_v1_corollary}
For every $N \geq 2$, the Ricci curvature of the Zeitlin metric on $SU(N)$ is strictly positive in the subspace $V_1$.
\end{corollary}

\subsection{Conjectured identities and numerical results}

To understand what happens in the eigenspaces $V_{\ell}$ for $\ell > 1$ as $N \rightarrow \infty$, we must work with Wigner $6j$ symbols, whose explicit combinatorial expressions are extremely long. In the course of computing Ricci curvature, we have observed certain recurring patterns involving sums of these symbols, which led us to formulate some new conjectural identities. 

To begin, for fixed $N$, we introduce the abbreviated notation below for certain $6j$ symbols that appear frequently throughout the paper:
\begin{equation}\label{eq_sixj_abbrev}
\mathcal{W}^{ij\ell} = \sixJ{i}{j}{\ell}, \qquad \mathcal{W}_j^i = \left\{\begin{matrix}
i & \frac{N-1}{2} & \frac{N-1}{2}
\vspace{0.15cm} \\ 
j & \frac{N-1}{2} & \frac{N-1}{2}
\end{matrix}
\right\}.
\end{equation}
The following summation formulas are well known from the theory of angular momentum in quantum mechanics, where $6j$ symbols play a fundamental role (see \cite{Varshalovich1988quantum}). 
\begin{align}
\sum\limits_{i = 0}^{N-1} (2i+1)(\mathcal{W}^{ij\ell})^2
&= \frac{1}{N}, \label{eq_sum_6j}
\\
\sum\limits_{i = 0}^{N-1}
(-1)^i(2i+1)(\mathcal{W}^{ij\ell})^2
&= (-1)^{N+1}\mathcal{W}_j^{\ell}, \label{eq_sum_6j_2}
\\ 
\sum\limits_{i = 0}^{N-1} (-1)^i(2i+1)
\mathcal{W}_j^i &= (-1)^{N+1}\delta_{j0}N.
\label{eq_sum_alternating}
\end{align}

We now list the identities observed in our computations, which seem to extend the  above in a natural way.

\begin{conjecture}\label{conj_new_formulas}
For all $1 \leq j, \ell \leq N-1$, we have
\begin{align}
\sum\limits_{i = 1}^{N-1} \frac{(2i+1)}{\lambda_{i}}
(\mathcal{W}^{ij\ell})^2 &= \frac{1}{N|j - \ell|(j + \ell + 1)}, \qquad (j\neq \ell), \label{intro_alt_lambda_small} \\
\sum\limits_{i = 1}^{N-1} (-1)^{i} \lambda_{i}(2i+1)
(\mathcal{W}^{ij\ell})^2 &=(-1)^{N+1}(\lambda_j + \lambda_{\ell})\mathcal{W}_j^{\ell}, \label{intro_alt_lambda} \\
\sum\limits_{i = 1}^{N-1} \lambda_{i}(2i+1)
(\mathcal{W}^{ij\ell})^2 &= \frac{(N^2-1)\big(\lambda_j+\lambda_{\ell}\big) - 2\lambda_j\lambda_{\ell}}{N(N^2-1)}, \label{intro_non_alt_lambda} \\
\sum\limits_{i = 1}^{N-1}\frac{2i+1}{\lambda_i}\left(\frac{1}{N} + (-1)^{i+j+N}\mathcal{W}_j^i\right) &= \frac{2H_j}{N},\label{intro_eq_harmonic}
\end{align}
where $\lambda_i = i(i+1)$ as before and $H_j = 1 + 1/2 + \cdots + 1/j$ is the $j^{\text{th}}$ harmonic number.
\end{conjecture}

These identities have been numerically verified up to $N = 2048$, which already covers the range relevant for practical simulations using the Zeitlin model. Note that since the square of a Wigner $6j$ symbol is a rational number, numerical verification is exact here.

For the values of $N$ where Conjecture \ref{conj_new_formulas} holds, we are able to get a closed form expression for the negative part of Ricci, $r_{\ell}(N)^-$, and an explicit upper bound for $r_{\ell}(N)^+$. If there were identities like \eqref{intro_alt_lambda_small} involving $(-1)^i$ or the case $j=\ell$, we would get an exact formula for $r_{\ell}(N)^+$, but we have not found these.

\begin{theorem}\label{theorem_3}
Assuming the identities in Conjecture \ref{conj_new_formulas}, the positive and negative parts of the Ricci curvature in the $V_{\ell}$ subspace satisfy: 
\begin{gather}\label{eq_ricci_thm3}
\begin{split}
r_{\ell}(N)^- &= \frac{2(H_{\ell}-1)}{\hbar_N^2}
=
2(H_{\ell}-1)\left(\frac{N^2-1}{4}\right), \\
r_{\ell}(N)^+
&\leq
(4H_{\ell} + 2\ell+1) \left(\frac{N^2-1}{4}\right). 
\end{split}
\end{gather}
\end{theorem}
Remarkably, this formula -- derived as a  complicated sum involving Wigner $6j$ 
 symbols -- factors into a product of two terms, one depending only on $\ell$ and the other only on $N$. 

In light of Theorem \ref{theorem_3}, it is natural to introduce the averaged Ricci curvatures, based on dividing by $\dim(SU(N)) = N^2-1$:
\begin{equation}
\widetilde{r}_{\ell}(N) = 
\frac{r_{\ell}(N)}{N^2-1}, 
\qquad \widetilde{r}_{\ell}(N)^{\pm} = \frac{r_{\ell}(N)^{\pm}}{N^2-1}.
\end{equation}
Numerical computations up to $N = 2048$ indicate that $\widetilde{r}_{\ell}(N)$ is not only bounded in $N$, as expected from Theorem \ref{theorem_3}, but in fact converges to the limit of $\widetilde{r}_{\ell}(N)^-$ as $N$ goes to infinity -- see Figure~\ref{fig:r_ell_1_23_N} below. This is equivalent to the statement that $\widetilde{r}_{\ell}(N)^+ \rightarrow 0$ as $N \rightarrow \infty$.

\begin{figure}[H]
    \centering
    \includegraphics{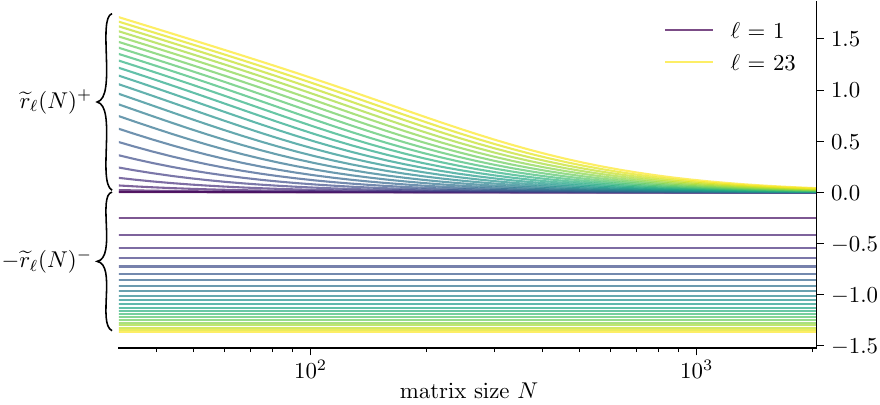}
    \caption{Positive $\widetilde{r}_\ell(N)^+$ and negative $\widetilde{r}_\ell(N)^-$ parts of the average Ricci curvature values in the Zeitlin model, for $\ell=1,\ldots,23$ and for matrix sizes $N$ from 32 to 2048 (in log-scale). 
    The negative part is independent of $N$, whereas the positive part seemingly converge to zero.}\label{fig:r_ell_1_23_N}
\end{figure}

\begin{figure}[H]
    \centering
    \includegraphics{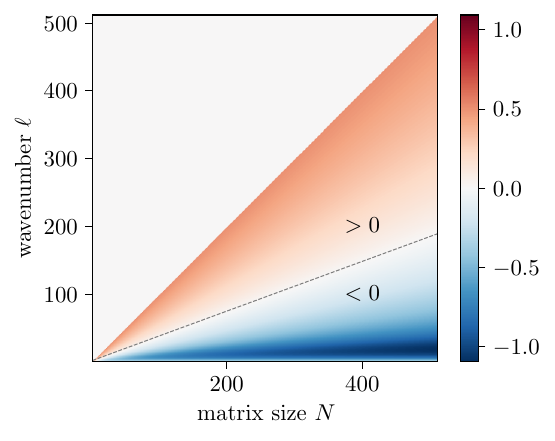}
    \caption{Average Ricci curvature values $\widetilde{r}_{\ell}(N)$ in the Zeitlin model for varying matrix sizes $N$. For a fixed $\ell$, the averaged Ricci curvature eventually becomes negative as $N$ grows, in accordance with Conjecture~\ref{conjecture_ricci}.
    On the other hand, as soon as $\ell>N/3$ the Ricci curvature in the Zeitlin model is positive. The slope $1/3$ is not clearly predicted even in our conjectured formulas, but is quite apparent graphically.}\label{fig:r_ell_N}
\end{figure}

\begin{center} 
\scalebox{0.9}{
\begin{tabular}{ |c|c|c|c|c|c|c|c|c|c|c|c|c|c|c|c|c| } 
 \hline
   \diagbox[innerwidth=1.8em]{$\ell$}{$N$}  & $3$ & $4$ & $5$ & $6$ & $7$ & $8$ & $9$ & $10$ & $11$ & $12$ & $13$ & $14$ & $15$ & $16$ \\ 
 \hline
 $1$  & \p & \p & \p & \p & \p & \p & \p & \p & \p & \p & \p & \p & \p & \p \\ 
 \hline
 $2$  & \p & \p & \p & \p & \n & \n & \n & \n & \n & \n & \n & \n & \n & \n \\ 
 \hline
 $3$  & & \p & \p & \p & \p & \p & \n & \n & \n & \n & \n & \n & \n & \n \\ 
 \hline
 $4$  & & & \p & \p & \p & \p & \p & \p & \n & \n & \n & \n & \n & \n \\ 
 \hline
 $5$  & & & & \p & \p & \p & \p & \p & \p & \p & \n & \n & \n & \n \\ 
 \hline
 $6$  & & & & & \p & \p & \p & \p & \p & \p & \p & \p & \p & \n \\ 
 \hline
 $7$  & & & & & & \p & \p & \p & \p & \p & \p & \p & \p & \p \\ 
 \hline
 $8$  & & & & & & & \p & \p & \p & \p & \p & \p & \p & \p \\ 
 \hline
 $9$  & & & & & & & & \p & \p & \p & \p & \p & \p & \p \\ 
 \hline
 $10$  & & & & & & & & & \p & \p & \p & \p & \p & \p \\ 
 \hline
 $11$  & & & & & & & & & & \p & \p & \p & \p & \p \\ 
 \hline
 $12$  & & & & & & & & & & & \p & \p & \p & \p \\ 
 \hline
 $13$  & & & & & & & & & & & & \p & \p & \p \\ 
 \hline
 $14$  & & & & & & & & & & & & & \p & \p \\ 
 \hline
 $15$  & & & & & & & & & & & & & & \p \\ 
 \hline
\end{tabular}}
\end{center}

\medskip

The table above shows the sign of the Ricci curvature in the direction of each eigenspace $V_\ell$ for $SU(N)$, with $\p$ and $\n$ denoting positive and negative curvature, respectively. The pattern mirrors the one observed in Figure \ref{fig:r_ell_N}, and persists for much higher $\ell$ and $N$. One can also observe the transition pattern $\ell/N \approx 1/3$ in this table.

\bigskip
Summarizing the above discussion and numerical results, we expect the following to hold.

\begin{conjecture}\label{conjecture_ricci}
For each fixed $\ell > 1$, the averaged Ricci curvature $\widetilde{r}_{\ell}(N)$ of the Zeitlin metric on $SU(N)$ becomes negative for sufficiently large $N$, and in the limit
\begin{equation}
\widetilde{r}_{\ell}(N) \rightarrow -\frac{H_{\ell}-1}{2}, \quad \text{as $N \rightarrow \infty$}, 
\end{equation}
where $H_{\ell}$ is the $\ell^{\text{th}}$ harmonic number.
\end{conjecture}

A possible strategy toward proving this conjecture begins with establishing the identities stated in Conjecture \ref{conj_new_formulas} through combinatorial methods. Once these identities are proved, the asymptotic behavior in Conjecture \ref{conjecture_ricci} would follow by deriving a  
sharper upper bound for $r_{\ell}(N)^+$ in order to prove that $\lim_{N\to\infty} \tilde{r}_{\ell}(N)^+=0$ for each $\ell$.

\section{Ricci curvature of the Zeitlin metric}\label{sec_ricci_finite_N}

In this section, we begin by reviewing the quantization framework underlying the Zeitlin model. Next, we derive a few lemmas that are useful for curvature computations and prove Theorem \ref{main_thm}.

\subsection{Quantization}

The incompressible Euler equations can be viewed, following Arnold’s framework, as geodesic equations on the infinite-dimensional Lie group $\mathrm{Diff}_\mu(S^2)$ of area-preserving diffeomorphisms of the sphere. The corresponding Lie algebra consists of divergence-free vector fields, which may be identified with Hamiltonian vector fields on the two-sphere. Passing to the vorticity formulation, one takes the curl of the first equation in \eqref{eq_euler}, yielding a scalar equation for the vorticity $\omega = \mathrm{curl}\, u$:
\begin{equation}\label{eq_euler_vorticity}
    \partial_t\omega + \{\psi, \omega\} = 0, \qquad \Delta \psi = \omega,
\end{equation}
where $\psi$ is the stream function, related to the velocity $u$ through the skew-gradient $u = \nabla^\perp\psi$, and $\{\cdot, \cdot\}$ denotes the Poisson bracket induced by the area form on the sphere. This formulation replaces the Lie algebra  $\mathrm{Diff}_{\mu, \mathrm{ex}}(S^2)$ of Hamiltonian vector fields with the Poisson algebra $\big(C^\infty_0(S^2), \{\cdot, \cdot\}\big)$ of smooth, mean-zero functions on $S^2$.

Quantization theory provides an approximation of this infinite-dimensional Poisson algebra by a sequence of matrix Lie algebras $\mathfrak{su}(N)$, each equipped with a distinguished right-invariant metric and scaled matrix commutators. Specifically, as $N \to \infty$, the algebras $\big(\mathfrak{su}(N), \frac{1}{\hbar_N}[\cdot,\cdot]\big)$, with $\hbar_N = 2/\sqrt{N^2 - 1}$ and the metric $g_N$ in \eqref{eq_zeitlin_metric}, converge weakly—via the $L^2$-adjoint of a quantization operator \( \mathsf{T}_N \colon C^\infty_0(S^2) \to \mathfrak{su}(N) \)—to the Poisson algebra \( (C^\infty_0(S^2), \{\cdot, \cdot\}) \) (see Charles and Polterovich~\cite{ChPo2018}). 
This approximation is the \emph{Zeitlin model} for the two-sphere. Zeitlin's discretization replaces the continuous vorticity equation \eqref{eq_euler_vorticity} with a finite-dimensional analogue:
\begin{equation}\label{eq:euler-zeitlin}
    \dot{W} + \frac{1}{\hbar}[P, W] = 0, \qquad \Delta_N P = W,
\end{equation}
with $P, W$ curves in the Lie algebra $\mathfrak{su}(N)$. The equations \eqref{eq:euler-zeitlin} retain the Lie–Poisson structure of the original system, and describe geodesics on \( SU(N) \) with respect to the right-invariant Riemannian metric \eqref{eq_zeitlin_metric}, derived from kinetic energy. 

In what follows, we collect a few lemmas that will be used throughout the paper, in particular in the proof of Theorem \ref{main_thm}.

\begin{lemma}\label{covderivprop}
The operator $\adstar$ defined by $\langle \adstar_uv, w\rangle = \langle v, \ad_uw\rangle$ for all $u,v,w\in \mathfrak{su}(N)$ is given by 
\begin{equation}\label{adstar}
\adstar_uv = -\Delta^{-1}\ad_u\Delta v,
\end{equation}
where $\Delta$ is the quantized Laplacian, defined relative to the quantization basis $\{T_{\ell m}\}$ by
\begin{equation}
\Delta T_{\ell m} = -\ell(\ell+1)T_{\ell m}.
\end{equation}
\end{lemma}

\begin{proof}
A straightforward computation using bi-invariance of $(\cdot, \cdot)$ and  $\langle u, v\rangle = -\tfrac{1}{N}(u, \Delta v)$. 
\end{proof}
The weights $\alpha_{\ell}$ relating the two metrics \eqref{eq_zeitlin_metric} and \eqref{eq_biinvariant_metric} are 
\begin{equation}\label{definition_alpha}
\langle T_{\ell m}, T_{\ell m} \rangle = \alpha_{\ell}\big(T_{\ell m}, T_{\ell m}\big), \quad \text{where} \quad \alpha_{\ell} = \frac{\ell(\ell+1)}{N}.
\end{equation}
The first ingredient towards the proof of Theorem \ref{main_thm} is a formula for Ricci curvature in terms of  $\weight_{\ell}$. Our starting point is Arnold's sectional curvature formula for a left- or right-invariant metric on any Lie group, which can be written in the form 
\begin{equation}\label{arnoldsectional}
K(u,v) = \frac{\frac{1}{4} \| \adstar_uv + \adstar_vu + \ad_uv \|^2 -  \langle \adstar_uv + \ad_uv, \ad_uv\rangle - \langle \adstar_uu, \adstar_vv\rangle}{\|u\|^2 \|v\|^2-\langle u,v\rangle^2}.
\end{equation}

Using \eqref{arnoldsectional}, we can derive the following expression for Ricci curvature. A similar formula was given in \cite{krishnan2021diagonal} in terms of an orthonormal basis, so we omit its proof. 
\begin{lemma}\label{simplifiedricci}
For any vectors $u\in V_{\ell}$ and $v\in V_{\ell'}$,  
the Ricci curvature of \eqref{eq_zeitlin_metric} is given by 
\begin{equation}\label{riccisymmetrized}
\Ric(u,v) = -\sum_{k=1}^{N-1} \sum_{k'=1}^{N-1} \frac{\weight_{\ell}\weight_{\ell'} - (\weight_{k}-\weight_{k'})^2}{4\weight_{k}\weight_{k'}} \Tr
{P_{k}\ad_u P_{k'}\ad_v P_{k}}.
\end{equation}
\end{lemma}

\subsection{Wigner symbols}

Formula \eqref{riccisymmetrized} depends on computing the traces of certain operators. On $SU(N)$, these traces can be computed directly in terms of the Wigner symbols that give the structure constants relative to the special basis $\{T_{\ell m}\}$. Following Zeitlin’s construction (\cite{Ze1991}, \cite{Ze2004}), we now recall the explicit form of these structure constants.

The commutator between $T_{\ell m}$ and $T_{\ell' m'}$ in $\mathfrak{su}(N)$ is given by
\begin{equation}
[T_{\ell m}, T_{\ell' m'}]
=
\sum\limits_{\ell''=1}^{N-1}
~\sum\limits_{m'' = -\ell''}^{\ell''} C_{\ell m\ell' m'}^{\ell'' m''(N)} T_{\ell'' m''}
\end{equation}
with structure constants 
\begin{gather}\label{eq_coefficients}
\begin{split}
C_{\ell m\ell' m'}^{\ell'' m''(N)} &= S_{\ell,\ell',\ell''} \Big(1 - (-1)^{\ell + \ell' + \ell''}\Big)(-1)^{m''+1}
\left(\begin{matrix}
\ell & \ell' & \ell'' \\
m & m' & m''
\end{matrix}
\right)
\left\{\begin{matrix}
\ell & \ell' & \ell'' \\
\frac{N-1}{2} & \frac{N-1}{2} & \frac{N-1}{2}
\end{matrix}
\right\},
\end{split}
\end{gather}
where:
\begin{enumerate}
    \item $S_{\ell,\ell',\ell''} = \sqrt{2\ell+1}\sqrt{2\ell'+1}\sqrt{2\ell''+1}$, so this coefficient is never zero;
    \item $(:::)$ denotes the Wigner $3j$ symbol;
    \item $\{:::\}$ denotes the Wigner $6j$ symbol.
\end{enumerate}

The Lie algebra $\mathfrak{su}(N)$ is split into subspaces (only $V_1$ is a Lie subalgebra) according to
$$V_{\ell} := \mathrm{span} \,\{ T_{\ell m} \, : \, |m| \leq \ell \}, \qquad 1 \leq \ell \leq N-1.$$
The coefficients $C_{\ell m\ell' m'}^{\ell'' m''(N)}$ in \eqref{eq_coefficients} are zero in the following special cases:
\begin{enumerate}
    \item If $\ell + \ell' + \ell''$ is even; or
    \item If $\ell,\ell',\ell''$ cannot be the sides of a triangle; e.g., if $|\ell-\ell'| > \ell''$; 
    \item If $m + m' + m'' \neq 0$.
\end{enumerate}
The last two properties come from selection rules of the Wigner $3j$ symbols (\cite{Varshalovich1988quantum}). Other zeros are possible, since it is an open problem to parametrize all zeros of Wigner $6$j-symbols.  We now have the necessary tools to prove Theorem \ref{main_thm}.

\begin{proof}[Proof of Theorem~\ref{main_thm}]
The key problem is to evaluate the trace term in formula \eqref{riccisymmetrized}. To do this, fix $k, k'$ and take the bases $\{T_{k m}\}_{|m| \leq k}$ of $V_{k}$ and $\{T_{k'm'}\}_{|m'| \leq k'}$ of $V_{k'}$, orthonormal in the bi-invariant metric \eqref{eq_biinvariant_metric}. Then, for any $u$ and $v$ we have
\begin{align}
\Tr{P_{k} \ad_{u} P_{k'}\ad_{v} P_{k}} 
&= \sum_{m=-k}^{k} \big( T_{k m}, \ad_{u} P_{k'} \ad_{v} T_{k m}\big) = -\sum_{m=-k}^{k} \big( \ad_{u} T_{k m}, P_{k'} \ad_{v} T_{k m} \big) \nonumber \\
&= -\sum_{m=-k}^{k} \sum_{m'=-k'}^{k'} \big( \ad_{u} T_{k m}, T_{k' m'}\big) \big( \ad_{v} T_{k m}, T_{k' m'}\big). \label{eq_big_trace}
\end{align}
Next, we specialize to basis vectors $u = T_{\ell r}$ and $v=T_{\ell' r'}$, orthonormal in the bi-invariant metric \eqref{eq_biinvariant_metric}, with $\lvert r\rvert \le \ell$ and $\lvert r'\rvert \le \ell'$. 

\begin{align*}
&-\hbar_N^2\Tr{P_{k}\ad_{u} P_{k'}\ad_{v} P_{k}} = \hbar_N^2\sum_{m=-k}^{k} \sum_{m'=-k'}^{k'} \big( \ad_{T_{\ell r}} T_{k m}, T_{k'm'}\big) \big( \ad_{T_{\ell' r'}} T_{k m}, T_{k'm'} \big) 
\\
&\qquad= \sum_{m=-k}^{k} \sum_{m'=-k'}^{k'} C_{\ell r k m}^{k' m'} C_{\ell' r' k m}^{k' m'}
\\
&\qquad= \sum_{m=-k}^{k} \sum_{m'=-k'}^{k'} S_{\ell,k,k'} \Big(1 - (-1)^{\ell+k+k'}\Big)(-1)^{m'+1}
\left(\begin{matrix}
\ell & k & k' \\
r & m & m'
\end{matrix}
\right)
\left\{\begin{matrix}
\ell & k & k' \\
\frac{N-1}{2} & \frac{N-1}{2} & \frac{N-1}{2}
\end{matrix}
\right\} \times
\\
&\qquad\qquad\qquad
S_{\ell',k,k'} \Big(1 - (-1)^{\ell'+k+k'}\Big)(-1)^{m'+1}
\left(\begin{matrix}
\ell' & k & k' \\
r' & m & m'
\end{matrix}
\right)
\left\{\begin{matrix}
\ell' & k & k' \\
\frac{N-1}{2} & \frac{N-1}{2} & \frac{N-1}{2}
\end{matrix}
\right\} \\ 
&\qquad= S_{\ell,k,k'}S_{\ell',k,k'} \Big(1 - (-1)^{\ell+k+k'}\Big) \Big(1 - (-1)^{\ell'+k+k'}\Big)  \times
\\
&\qquad\qquad\qquad \left\{\begin{matrix}
\ell & k & k' \\
\frac{N-1}{2} & \frac{N-1}{2} & \frac{N-1}{2}
\end{matrix}
\right\}  \left\{\begin{matrix}
\ell' & k & k' \\
\frac{N-1}{2} & \frac{N-1}{2} & \frac{N-1}{2}
\end{matrix}
\right\} 
\sum_{m=-k}^{k} \sum_{m'=-k'}^{k'} 
\left(\begin{matrix}
\ell & k & k' \\
r & m & m'
\end{matrix}
\right)
\left(\begin{matrix}
\ell' & k & k' \\
r' & m & m'
\end{matrix}
\right).
\end{align*}
Here we note that most of the terms in the structure constant $C^{k' m'}_{\ell r k m}$ do not depend on $m$ or $m'$, so they can be pulled out of the summation, and we are left with a standard orthogonality result for the $3j$-symbols (see \cite{Varshalovich1988quantum}):
$$\sum_{m=-k}^{k} \sum_{m'=-k'}^{k'} 
\left(\begin{matrix}
\ell & k & k' \\
r & m & m'
\end{matrix}
\right)
\left(\begin{matrix}
\ell' & k & k' \\
r' & m & m'
\end{matrix}
\right) = \frac{\delta_{\ell\ell'} \delta_{rr'}}{2\ell+1} \qquad \text{if $\lvert k-k'\rvert \le \ell\le k+k'$}.
$$
Plugging this result in, we get 
\begin{equation}\label{trace_form}
-\hbar_N^2 \Tr{P_{k}\ad_{T_{\ell r}} P_{k'}\ad_{T_{\ell' r'}} P_{k}} = \delta_{\ell\ell'} \delta_{rr'} \,
\frac{S_{\ell,k,k'}^2}{2\ell+1} \Big(1 - (-1)^{\ell+k+k'}\Big)^2
\left\{\begin{matrix}
\ell & k & k' \\
\frac{N-1}{2} & \frac{N-1}{2} & \frac{N-1}{2}
\end{matrix}
\right\}^2,
\end{equation}
which simplifies further to the statement 
\begin{equation*}
-\Tr{P_{k}\ad_{T_{\ell r}} P_{k'}\ad_{T_{\ell r}} P_{k}} = 
\frac{4 (2k+1)(2k'+1)}{ \hbar_N^2 } \left\{\begin{matrix}
\ell & k & k' \\
\frac{N-1}{2} & \frac{N-1}{2} & \frac{N-1}{2}
\end{matrix}
\right\}^2
\end{equation*}
if $\ell+k+k'$ is odd, for any $r$ with $\lvert r\rvert \le \ell$.  

Since the right-hand side of \eqref{trace_form} vanishes whenever the indices in the top row of the $6j$ symbol are even, we get from \eqref{riccisymmetrized} the formula
\begin{gather*}
\begin{split}
\Ric(T_{\ell r},T_{\ell' r'}) &= \frac{\delta_{\ell\ell'}\delta_{rr'}}{\hbar_N^2} \sum\limits_{\substack{k, k' = 1 \\ \ell+k+k'\;\mathrm{odd}}}^n \frac{\weight_{\ell}\weight_{\ell'} - (\weight_{k}-\weight_{k'})^2}{\weight_{k}\weight_{k'}} 
\, (2k+1)(2k'+1)  \left\{\begin{matrix}
\ell & k & k' \\
\frac{N-1}{2} & \frac{N-1}{2} & \frac{N-1}{2}
\end{matrix}
\right\}^2
\end{split}
\end{gather*}

Now for any vector $u\in V_{\ell}$, we can write $u=\sum_{r=-\ell}^{\ell} u_r T_{\ell r}$ in terms of the basis $T_{\ell r}$ which is orthonormal in the bi-invariant metric. Thus the weighted metric will have $\langle u,u\rangle = \alpha_{\ell} \big( u,u\big)=\alpha_{\ell}\sum_{r=-\ell}^{\ell} u_r^2$, while the Ricci curvature will be 
\begin{align*}
\Ric(u,u) &= \sum_{r=-\ell}^{\ell} \sum_{r'=-\ell}^{\ell} u_ru_{r'}\Ric(T_{\ell r},T_{\ell r'}) \\
&= \sum_{r=-\ell}^{\ell} \frac{u_r^2}{\hbar_N^2} \sum\limits_{\substack{k, k' = 1 \\ \ell+k+k'\;\mathrm{odd}}}^n \frac{\weight_{\ell}^2 - (\weight_{k}-\weight_{k'})^2}{\weight_{k}\weight_{k'}} 
\, (2k+1)(2k'+1)  \left\{\begin{matrix}
\ell & k & k' \\
\frac{N-1}{2} & \frac{N-1}{2} & \frac{N-1}{2}
\end{matrix}
\right\}^2 \\
&= r_{\ell}(N) \,\langle u,u\rangle
\end{align*}
in each subspace $V_{\ell}$, 
where 
\begin{equation*}
    r_{\ell}(N) := \frac{1}{\hbar_N^2} \sum\limits_{\substack{k, k' = 1 \\ \ell+k+k'\;\mathrm{odd}}}^n \frac{\weight_{\ell}^2 - (\weight_{k}-\weight_{k'})^2}{\weight_{k}\weight_{k'}\weight_{\ell}} 
\, (2k+1)(2k'+1)  \left\{\begin{matrix}
\ell & k & k' \\
\frac{N-1}{2} & \frac{N-1}{2} & \frac{N-1}{2}
\end{matrix}
\right\}^2,
\end{equation*} 
as desired. Note that the extra factor of $\alpha_{\ell}$ in the denominator is a consequence of expressing the Ricci curvature as a multiple of the weighted metric, not the bi-invariant metric.

Recalling that $\alpha_{\ell} = \lambda_{\ell}/N$ for each $\ell$ in the Zeitlin metric \eqref{eq_zeitlin_metric}, we can express this in terms of the original parameters $\lambda_{\ell}=\ell(\ell+1)$ via 
\begin{equation}\label{properscaledric}
r_{\ell}(N) := \frac{N}{\hbar_N^2} \sum\limits_{\substack{k, k' = 1 \\ \ell+k+k'\;\mathrm{odd}}}^n \frac{\lambda_{\ell}^2 - (\lambda_{k}-\lambda_{k'})^2}{\lambda_{k}\lambda_{k'}\lambda_{\ell}} 
\, (2k+1)(2k'+1)  \left\{\begin{matrix}
\ell & k & k' \\
\frac{N-1}{2} & \frac{N-1}{2} & \frac{N-1}{2}
\end{matrix}
\right\}^2.
\end{equation}
Finally, note that from \eqref{properscaledric} we can split the Ricci curvature into a positive and a negative term
$r_{\ell}(N) = r_{\ell}(N)^+-r_{\ell}(N)^-$ where $r_{\ell}(N)^+$ and $r_{\ell}(N)^-$ are given in \eqref{eq_ricci_intro}. This concludes the proof of Theorem \ref{main_thm}.
\end{proof}

\section{Asymptotic analysis and conjectured formulas}\label{section_asymptotic}

We now turn to the asymptotic regime, with the goal of understanding the limiting behavior of the Ricci eigenvalues $r_\ell(N)$ as $N \rightarrow \infty$. As in \eqref{eq_sixj_abbrev}, 
we write:
\begin{equation*}
\mathcal{W}^{ij\ell} = \sixJ{i}{j}{\ell}, \qquad \mathcal{W}_j^i = \left\{\begin{matrix}
i & \frac{N-1}{2} & \frac{N-1}{2} \\ 
j & \frac{N-1}{2} & \frac{N-1}{2}
\end{matrix}
\right\}
\end{equation*}
for these Wigner $6j$ symbols. Formula \eqref{eq_sum_6j} immediately gives the upper bound
\begin{equation}\label{eq_6j_upper_bound}
\big(\mathcal{W}^{ij\ell}\big)^2 \leq \frac{1}{N(2\ell+1)}.
\end{equation}

\begin{remark}
In the case where $a+b+c$ is \emph{even}, the asymptotic formula (see \cite{brussaard1957classical})
\begin{equation}\label{eq_clebschgordan}
\big(\mathcal{W}^{ij\ell}\big)^2 \simeq \frac{1}{2(N-1)(2\ell+1)}\big(C^{\ell 0}_{i0j0}\big)^2
\end{equation}
where $C^{\ell 0}_{i0j0}$ is the Clebsch-Gordan coefficient, shows that as far as $N$ is concerned, the upper bound \eqref{eq_6j_upper_bound} is essentially sharp. On the other hand, if $i+j+\ell$ is \emph{odd}, then the coefficients $C^{\ell 0}_{i0j0}$ vanish and the proof of \eqref{eq_clebschgordan} given in \cite{brussaard1957classical} is no longer valid. Indeed, numerically it seems that a sharper bound than  \eqref{eq_6j_upper_bound} is possible in the case of  $i+j+\ell$ odd, which would help with Conjecture \ref{conjecture_ricci}, but this requires exploiting this parity assumption somehow.  
\end{remark}

Concerning the sign of Ricci curvature, we first prove Corollary \ref{positive_v1_corollary}, which states that Ricci curvature is always positive in the $V_1$ subspace.

\begin{proof}[Proof of Corollary \ref{positive_v1_corollary}]
We first claim that the negative part of Ricci in the $V_1$ subspace, namely $r_1^-(N)$, is zero for every $N$. From Theorem \ref{main_thm}, we have
\begin{equation}\label{eq_v1_negative_ricci}
r_{1}(N)^- = \frac{N}{\hbar_N^2}\sum\limits_{\substack{i, j = 1 \\ i+1+j\;\mathrm{odd}}}^{N-1}
\frac{
(\lambda_i-\lambda_j)^2
(2i+1)(2j+1)
}
{\lambda_i\lambda_j\lambda_{1}}\left(\mathcal{W}^{ij1}\right)^2.
\end{equation}
Due to selection rules, the upper row $i, j, 1$ of the $6j$ symbol $\mathcal{W}^{ij1}$ must satisfy $|i-j| \leq 1$. At the same time, our Ricci formula forces $i+j+1$ to be odd, so that $i+j$ is even, which means that only $i = j$ is allowed in the sum for $r_1^-(N)$. Thus, \eqref{eq_v1_negative_ricci} vanishes due to the $(\lambda_i-\lambda_j)^2$ term.

On the other hand, the positive part
\begin{equation}\label{eq_v1_positive_ricci}
r_{1}(N)^+ = \frac{N}{\hbar_N^2}\sum\limits_{\substack{i, j = 1 \\ i+1+j\;\mathrm{odd}}}^{N-1}
\frac{
\lambda_{1}
(2i+1)(2j+1)
}
{\lambda_i\lambda_{j}}\left(\mathcal{W}^{ij1}\right)^2
\end{equation}
never vanishes, because in the special case $i = j = 1$, we have (cf. \cite{Varshalovich1988quantum})
\begin{equation}
\left(\mathcal{W}^{ij1}\right)^2 = \frac{2}{3N(N^2-1)},
\end{equation}
ensuring that the sum \eqref{eq_v1_positive_ricci} is strictly positive.
\end{proof}

\noindent In the remainder of this section, we make use of the identities listed in Conjecture \ref{conj_new_formulas}, and derive the necessary estimates to prove Theorem \ref{theorem_3}. The proof is broken down into two parts: an upper bound for $r_{\ell}(N)^+$ and an exact formula for $r_{\ell}(N)^-$. 

\subsection{Upper bound for $r_{\ell}(N)^+$}

We start with \eqref{eq_ricci_intro}, estimating
\begin{equation}
\frac{\hbar_N^2}{N\lambda_\ell} r_\ell^+(N)
=
\sum\limits_{\substack{i, k = 1 \\ i+k+\ell\;\mathrm{odd}}}^{N-1}
\frac{
(2i+1)(2k+1)
}
{\lambda_i\lambda_{k}}(\mathcal{W}^{i\ell k})^2
\le 
\sum_{i, k = 1}^{N-1}
\frac{
(2i+1)(2k+1)
}
{\lambda_i\lambda_{k}}(\mathcal{W}^{i\ell k})^2 
\end{equation}

Therefore, by \eqref{intro_alt_lambda_small} we get
\begin{equation}\label{rplusupper}
\hbar_N^2 r_\ell^+(N)
\le 
\sum\limits_{\substack{i = 1 \\ i\neq \ell}}^{N-1}
\frac{\lambda_{\ell} (2i+1)}{\lambda_i|i-\ell|(i+\ell+1)} 
+
N(2\ell+1)
\sum\limits_{\substack{k = 1 \\ k \;\mathrm{odd}}}^{N-1}
\frac{(2k+1)}{\lambda_{k}}\big(\mathcal{W}^{\ell\ell k}\big)^2
\end{equation}

The partial fraction decomposition of the first sum in \eqref{rplusupper} is 
\begin{align*}
\sum\limits_{\substack{i = 1 \\ i\neq \ell}}^{N-1}
\frac{\lambda_{\ell} (2i+1)}{\lambda_i|i-\ell|(i+\ell+1)}  &= -\sum_{i=1}^{\ell-1} \left( \frac{1}{i+\ell+1} + \frac{1}{i-\ell} - \frac{1}{i} - \frac{1}{i+1}\right)  \\
&\qquad\qquad
    + \sum_{i=\ell+1}^{N-1} \left( \frac{1}{i+\ell+1} + \frac{1}{i-\ell} - \frac{1}{i} - \frac{1}{i+1}\right) \\
    &= 2(H_{\ell-1}+H_{\ell}+H_{\ell+1}) - 1 -( H_{2\ell}+H_{2\ell+1}) \\
    &\qquad\qquad -(H_N + H_{N-1}) + H_{N+\ell} + H_{N-\ell-1} \\
    &\le 4H_{\ell}
\end{align*}
in terms of the harmonic numbers $H_{\ell} = \sum_{k=1}^{\ell} k^{-1}$. Meanwhile 
the second sum is estimated roughly using \eqref{eq_sum_6j}. Combining the two estimates gives 
$$ \hbar_N^2 r_\ell^+(N) \le 4H_{\ell}+2\ell+1.$$
We expect this is far from sharp, since numerically it approaches zero faster than this as $N\to\infty$; see Figure \ref{fig:r_ell_1_23_N}. To proceed further we would need identities for the unknown terms 
$$ \sum_{k=1}^{N-1} (-1)^k \, \frac{2k+1}{\lambda_k} \, (\mathcal{W}^{i\ell k})^2 \qquad \text{and}\qquad \sum\limits_{\substack{k = 1 \\ k \;\mathrm{odd}}}^{N-1}
\frac{(2k+1)}{\lambda_{k}}\big(\mathcal{W}^{\ell\ell k}\big)^2.
$$ 

\color{black}

\subsection{Formula for $r_{\ell}^-(N)$}

Starting from the definition of $r_{\ell}^-(N)$ in \eqref{eq_ricci_intro}, 
\begin{gather}
\begin{split}
\frac{\hbar_N^2}{N} \lambda_\ell r_\ell^{-}(N)
&=
\sum\limits_{\substack{i, k = 1 \\ i+\ell+k\;\mathrm{odd}}}^{N-1}
\frac{
\big(\lambda_i - \lambda_{k}\big)^2
(2i+1)(2k+1)
}
{\lambda_i\lambda_{k}}
(\mathcal{W}^{i\ell k})^2,
\end{split}
\end{gather}
we expand the $(\lambda_i-\lambda_{k})^2$ term and split the above sum as
\begin{gather}
\begin{split}
\sum\limits_{\substack{i, k = 1 \\ i+\ell+k\;\mathrm{odd}}}^{N-1}
\frac{
\big(\lambda_i - \lambda_{k}\big)^2
(2i+1)(2k+1)
}
{\lambda_i\lambda_{k}}(\mathcal{W}^{i\ell k})^2 = 2A - 2B
\end{split}
\end{gather}
where
\begin{gather}\label{eq_ABC}
\begin{split}
A &:= \sum\limits_{\substack{i, k = 1 \\ i+\ell+k\;\mathrm{odd}}}^{N-1}
\frac{
\lambda_{k}(2i+1)(2k+1)
}
{\lambda_i}(\mathcal{W}^{i\ell k})^2 
= \sum\limits_{\substack{i, k = 1 \\ i+\ell+k\;\mathrm{odd}}}^{N-1}
\frac{
\lambda_i(2i+1)(2k+1)
}
{\lambda_{k}}(\mathcal{W}^{i\ell k})^2, \\ 
B &:=
\sum\limits_{\substack{i, k = 1 \\ i+\ell+k\;\mathrm{odd}}}^{N-1}
(2i+1)(2k+1)(\mathcal{W}^{i\ell k})^2.
\end{split}
\end{gather}

We use the fact that $(1-(-1)^{i+\ell+k})$ vanishes if and only if $i+\ell+k$ is odd to remove the oddness constraint in the sum \eqref{eq_ABC}, and we get
\begin{gather*}
    \begin{split}
        2A&= \sum_{i,k=1}^{N-1} \frac{(1-(-1)^{i+\ell+k}) \lambda_k(2i+1)(2k+1)}{\lambda_i} (W^{i\ell k})^2 \\
        &= \sum_{i=1}^{N-1} \frac{2i+1}{\lambda_i} \sum_{k=1}^{N-1} \lambda_k(2k+1) (W^{i\ell k})^2 + \sum_{i=1}^{N-1} \frac{(-1)^{i+\ell+1}(2i+1)}{\lambda_i} \sum_{k=1}^{N-1} (-1)^k \lambda_k(2k+1) (W^{i\ell k})^2 \\
        &= \sum_{i=1}^{N-1} \frac{2i+1}{\lambda_i} \, \frac{(N^2-1)(\lambda_i+\lambda_{\ell}) - 2\lambda_i\lambda_{\ell}}{N(N^2-1)}  + \sum_{i=1}^{N-1} \frac{(-1)^{i+\ell+N}(2i+1) (\lambda_i+\lambda_{\ell})}{\lambda_i} \, W^{\ell}_i,
        \end{split}
        \end{gather*}
using formulas \eqref{intro_alt_lambda} and \eqref{intro_non_alt_lambda} from Conjecture \ref{conj_new_formulas}.
Now we split each term of the sum into those that involve $\lambda_i$ in the numerator and those that do not.
        \begin{gather}\label{Aderivation}
            \begin{split}
     2A   &= \sum_{i=1}^{N-1} (2i+1) \, \frac{(N^2-1-2\lambda_{\ell})}{N(N^2-1)}  +
        \sum_{i=1}^{N-1} \frac{2i+1}{\lambda_i} \, \frac{\lambda_{\ell}}{N} \\
        &\qquad\qquad+
        \sum_{i=1}^{N-1} (-1)^{i+\ell+N}(2i+1) \, W^{\ell}_i
        +
        \sum_{i=1}^{N-1} \frac{(-1)^{i+\ell+N}(2i+1) \lambda_{\ell}}{\lambda_i} \, W^{\ell}_i\\
        &= \frac{N^2-1-2\lambda_{\ell}}{N} + \lambda_{\ell} \sum_{i=1}^{N-1} \frac{2i+1}{\lambda_i} \left( \frac{1}{N} + (-1)^{i+\ell+N} W^{\ell}_i\right) \\
        &\qquad\qquad + (-1)^{\ell+N} \sum_{i=1}^{N-1} (-1)^i (2i+1) W^{\ell}_i \\
        &= \frac{N^2-1-2\lambda_{\ell}}{N} + \frac{2\lambda_{\ell} H_{\ell}}{N} + (-1)^{\ell+N+1} W^{\ell}_{0}
    \end{split}
\end{gather}
using \eqref{intro_eq_harmonic} from Conjecture 
\ref{conj_new_formulas} along with the identity  \eqref{eq_sum_alternating} 
in the form 
$$ \sum_{i=1}^{N-1} (-1)^i (2i+1) W^{\ell}_{i} = -W^{\ell}_0.$$

Now to evaluate $B$, notice that we do not need any of the conjectured identities, only the known ones. We write 
\begin{align*}
    2B &= \sum_{i,k=1}^{N-1} (1-(-1)^{i+\ell+k}) (2i+1)(2k+1) (\mathcal{W}^{i\ell k})^2 \\
    &= \sum_{i=1}^{N-1} (2i+1) \sum_{k=1}^{N-1} (2k+1) (\mathcal{W}^{i\ell k})^2 + \sum_{i=1}^{N-1} (-1)^{i+\ell+1} (2i+1)  \sum_{k=1}^{N-1} (-1)^k (2k+1) (\mathcal{W}^{i\ell k})^2 \\
    &= \sum_{i=1}^{N-1} (2i+1) \left( \frac{1}{N} - (\mathcal{W}^{i\ell 0})^2\right) + \sum_{i=1}^{N-1} (-1)^{i+\ell+1} (2i+1) \left( (-1)^{N+1}\mathcal{W}_i^{\ell} - (\mathcal{W}^{i\ell 0})^2\right),
\end{align*}
using formulas \eqref{eq_sum_6j}--\eqref{eq_sum_6j_2}, correcting for the fact that those sums start at $k=0$ rather than $k=1$.

Using the identities \eqref{eq_sum_6j}--\eqref{eq_sum_alternating}, this becomes 
\begin{gather}\label{Bfinal}
\begin{split}
    2B&= \frac{1}{N} \sum_{i=1}^{N-1} (2i+1)   + (-1)^{\ell+N} \sum_{i=1}^{N-1} (-1)^i (2i+1) \mathcal{W}_i^{\ell} \\
    &\qquad\qquad - \sum_{i=1}^{N-1} (2i+1) (\mathcal{W}^{i\ell 0})^2
    + \sum_{i=1}^{N-1} (-1)^{i+\ell} (2i+1) (\mathcal{W}^{i\ell 0})^2 \\
    &= \frac{N^2-1}{N}  + (-1)^{\ell+N} (-\mathcal{W}^{\ell}_0) -
    \sum\limits_{\substack{i = 1 \\ i+\ell \;\mathrm{odd}}}^{N-1} (2i+1) (\mathcal{W}^{i\ell 0})^2\\
    &= \frac{N^2-1}{N} + (-1)^{\ell+N+1} \mathcal{W}^{\ell}_0.
\end{split}
\end{gather}
Here the last sum vanishes since $\mathcal{W}^{i\ell 0}$ cannot satisfy the triangular relation if $(i+\ell)$ is odd.  
Subtracting \eqref{Bfinal} from \eqref{Aderivation}, we obtain 
\begin{equation}\label{final_rminus}
    \hbar_N^2 \lambda_{\ell} r_{\ell}^-(N) = 2A-2B = 2\lambda_{\ell}(H_{\ell}-1).
\end{equation} 

\noindent\textbf{Final formula and estimate on $r_\ell(N)$}. Putting everything together, we get 
\begin{gather}\label{eq_final_estimate}
\begin{split}
r_\ell(N) &= r_\ell^+(N) - r_\ell^-(N) \\
&\leq \frac{1}{\hbar_N^2}  
(4H_{\ell}+2\ell+1-2H_{\ell}+2) \\
&= \frac{(N^2-1)2H_{\ell}+2\ell+3}{4}.
\end{split}
\end{gather}
The lower bound 
$$ r_{\ell}(N) \ge -\frac{(N^2-1)(H_{\ell}-1)}{2}$$
can also be obtained by simply ignoring $r_{\ell}^+(N)$.  
This concludes the proof of Theorem \ref{theorem_3}.

\section{The homogeneous space $SU(N)/SO(3)$}
\label{homogeneous_section}

This section is devoted to the full computation of Ricci curvature for the homogeneous space $SU(N)/SO(3)$, leading to the proof of Theorem \ref{homogeneous_space_thm}. As explained in the introduction, this quotient space is natural from the point of view large scale, long-time behaviour of hydrodynamics on the sphere, serving as a model for $\mathrm{Diff}_{\mu}(S^2)/SO(3)$. 

Let $\mathfrak{h} = V_1$. This is a Lie subalgebra of $\mathfrak{su}(N)$, corresponding to the subgroup ${SO(3) \subseteq SU(N)}$. In the Zeitlin metric, its orthogonal complement is
\begin{equation}
\mathfrak{h}^{\perp} = V_2 \oplus \cdots \oplus V_{N-1}.
\end{equation}
The projection $\pi : SU(N) \rightarrow SU(N)/SO(3)$ is a Riemannian submersion, and by O'Neill's formulas, the sectional curvatures $K_G$ of the total space $G = SU(N)$ and $K_B$ of the base space $B = SU(N)/SO(3)$ are related by (see \cite{oneill1966thefundamental})
\begin{equation}\label{eq_submersion_curvature}
K_B(X, Y) = K_G(X, Y) + 3\frac{\|A_XY\|^2}{\|X \wedge Y\|^2}
\end{equation}
provided that $X, Y$ are horizontal vectors, linearly independent so that the normalizing factor $${\|X \wedge Y\|^2 := \|X\|^2\|Y\|^2 - \langle X, Y \rangle^2}$$
is not zero. The tensor $A$ is given by
\begin{equation}
A_XY = \frac{1}{2}\mathcal{V}\big(\mathcal{L}_XY\big)
\end{equation}
where $\mathcal{L}_XY$ is the Lie derivative of $Y$ in the direction of $X$, and  $\mathcal{V}$ is the projection onto the vertical subspace $V_1$. We begin by noting the following simplification, which will turn out to be very useful. 
\begin{lemma}\label{lemma_3_sum}
Let $C_{\ell,m,\ell',m'}^{\ell'',m''}$ denote the structure constants as given in \eqref{eq_coefficients}. Then, we have
\begin{equation}
\Big(C_{\ell,m,\ell,1-m}^{1,-1}\Big)^2 + 
\Big(C_{\ell,m,\ell,-m}^{1,0}\Big)^2 + 
\Big(C_{\ell,m,\ell,-m-1}^{1,1}\Big)^2 = \frac{12\ell(\ell+1)}{N(N^2-1)}
\end{equation}    
for all $1 \leq \ell \leq N-1$ and $|m| \leq \ell$. Note that the right-hand side does not depend on $m$.
\end{lemma}

\begin{proof}
Starting from \eqref{eq_coefficients}, first note that since $\ell''=1$ and $\ell'=\ell$, 
\begin{equation}
\Big(\sqrt{2\ell+1}\sqrt{2\ell'+1}\sqrt{2\ell''+1}\big(1-(-1)^{\ell+\ell+1}\big)\Big)^2 = 12(2\ell+1)^2.
\end{equation}
Next, we turn to the product of Wigner symbols
\begin{gather*}
\begin{split}
D_{\ell m\ell m'}^{1 m''} := 
\left(\begin{matrix}
\ell & \ell & 1 \\
m & m' & m''
\end{matrix}
\right)
\left\{\begin{matrix}
\ell & \ell & 1 \\
\frac{N-1}{2} & \frac{N-1}{2} & \frac{N-1}{2}
\end{matrix}
\right\},
\end{split}
\end{gather*}
both of which admit simple closed form expressions due to repeated indices and the presence of the $1$ in the upper row. They are given by (see \cite{Varshalovich1988quantum}, in particular section $9.5.4$ for the $6j$ symbol):
\begin{align*}
\left\{\begin{matrix}
\ell & \ell & 1 \\
\frac{N-1}{2} & \frac{N-1}{2} & \frac{N-1}{2}
\end{matrix}\right\}^2
&=
\frac{\ell(\ell+1)}{N(N^2-1)(2\ell+1)},
\quad 
&\left(\begin{matrix}
\ell & \ell & 1 \\
m & 1-m & -1
\end{matrix}\right)^2
&=
\frac{(\ell-m+1)(\ell+m)}{2\ell(\ell+1)(2\ell+1)}, \\
\left(\begin{matrix}
\ell & \ell & 1 \\
m & -m & 0
\end{matrix}\right)^2
&=
\frac{m^2}{\ell(\ell+1)(2\ell+1)},
\quad
&\left(\begin{matrix}
\ell & \ell & 1 \\
m & -1-m & 1
\end{matrix}\right)^2 &= \frac{(\ell-m)(1+\ell+m)}{2\ell(\ell+1)(2\ell+1)}.
\end{align*}
Note that the $3j$ symbols on the right column above vanish when $m = -\ell$ (top) and $m = \ell$ (bottom), so that we can ignore those two exceptional cases. From this, we have
\begin{gather*}
\begin{split}
\Big(D_{\ell, m,\ell, 1-m}^{1,-1}\Big)^2
&= 
\frac{(\ell-m+1)(\ell+m)}{2N(N^2-1)(2\ell+1)^2},
\qquad \Big(D_{\ell, m,\ell, -m}^{1,0}\Big)^2
= 
\frac{m^2}{N(N^2-1)(2\ell+1)^2},
\\
\Big(D_{\ell, m,\ell,-1-m}^{1,1}\Big)^2
&= 
\frac{(\ell-m)(1+\ell+m)}{2N(N^2-1)(2\ell+1)^2}.
\end{split}
\end{gather*}
Therefore, the squares of the coefficients in \eqref{eq_A_formula_2} are
\begin{alignat*}{4}
\Big(C_{\ell,m,\ell,1-m}^{1,-1}\Big)^2 &= 12(2\ell+1)^2\frac{(\ell-m+1)(\ell+m)}{2N(N^2-1)(2\ell+1)^2} ~&=~
&\frac{6(\ell-m+1)(\ell+m)}{N(N^2-1)},&
\\ 
\Big(C_{\ell,m,\ell,-m}^{1,0}\Big)^2 &=
12(2\ell+1)^2\frac{m^2}{N(N^2-1)(2\ell+1)^2} ~&=~
&\frac{12m^2}{N(N^2-1)},&
\\
\Big(C_{\ell,m,\ell,-m-1}^{1,1}\Big)^2 &=
12(2\ell+1)^2\frac{(\ell-m)(1+\ell+m)}{2N(N^2-1)(2\ell+1)^2} ~&=~ &\frac{6(\ell-m)(1+\ell+m)}{N(N^2-1)}.&
\end{alignat*}
Adding these up concludes the proof.
\end{proof}
The next lemma is the main tool we need in order to compute the Ricci curvature of the quotient space $SU(N)/SO(3)$ using O'Neill's formula 
\eqref{eq_submersion_curvature}. 

\begin{lemma}\label{lemma_A_XY}
For any fixed $(\ell,m)$ with $2 \leq \ell \leq N-1$ and $-\ell \leq m \leq \ell$, we have
\begin{align*}
\sum\limits_{\ell' = 2}^{N-1}
~\sum\limits_{m' = -\ell'}^{\ell'} 3\frac{\|A_{T_{\ell m}}T_{\ell'm'}\|^2}{\|T_{\ell m} \wedge T_{\ell'm'}\|^2} = \frac{9}{2\ell(\ell+1)},
\end{align*}
where in the above summation, we skip the term $(\ell', m') = (\ell, m)$, for which the summand on the left is not well-defined.
\end{lemma}

\begin{proof}
We expand the Lie bracket of $X = T_{\ell m}$ and $Y = T_{\ell'm'}$ according to  \eqref{eq_coefficients}, but since we are only interested in the vertical component, the condition $\ell'' = 1$ is imposed, which forces  $m''\in\{-1,0,1\}$, and we get 
\begin{gather}\label{eq_A_formula_1}
\begin{split}
A_{T_{\ell m}}T_{\ell' m'} &= \frac{1}{2}\mathcal{V} (\ad_{T_{\ell m}} T_{\ell'm'})  =-\frac{1}{2\hbar_N} 
[T_{\ell m}, T_{\ell' m'}] \\
&= -\frac{1}{2\hbar_N}\Big(C_{\ell m\ell' m'}^{1,-1}T_{1,-1} + C_{\ell m\ell' m'}^{1,0}T_{1,0} + C_{\ell m\ell' m'}^{1,1}T_{1,1}\Big).
\end{split}
\end{gather}
Selection rules for $3j$ symbols imply a dependence of $\ell', m'$ on $\ell, m$, so the only nonzero values are:
\begin{gather}\label{eq_A_formula_2}
\begin{split}
A_{T_{\ell m}}T_{\ell, 1-m}
&=
-\tfrac{1}{2\hbar_N}C_{\ell,m,\ell,1-m}^{1,-1}T_{1,-1} \qquad \text{for $m \neq -\ell$}, \\
A_{T_{\ell m}}T_{\ell, -m} 
&=
-\tfrac{1}{2\hbar_N}C_{\ell,m,\ell,-m}^{1,0}T_{1,0}, \\
A_{T_{\ell m}}T_{\ell, -1-m}
&=
-\tfrac{1}{2\hbar_N}C_{\ell,m,\ell,-m-1}^{1,1}T_{1,1} \qquad \text{for $m \neq \ell$}.
\end{split}
\end{gather}
Recall that the $\{ T_{\ell m} \}$ basis is orthonormal in the bi-invariant metric, but only orthogonal in the Zeitlin metric \eqref{eq_zeitlin_metric}, so we have \begin{equation*}
\| T_{\ell m} \wedge T_{\ell' m'} \|^2 = \|T_{\ell m}\|^2\|T_{\ell'm'}\|^2 = \frac{\ell(\ell+1)\ell'(\ell'+1)}{N^2}, \qquad \text{for $(\ell, m) \neq (\ell',m')$}.
\end{equation*}
In particular, $\|T_{1m}\|^2 = 2/N$ for $|m| \leq 1$. Thus, summing over all $(\ell', m') \neq (\ell, m)$ and using \eqref{eq_A_formula_2},
\begin{gather*}
\begin{split}
&\sum\limits_{\ell' = 2}^{N-1}
\sum\limits_{m' = -\ell'}^{\ell'} 3\frac{\|A_{T_{\ell m}}T_{\ell'm'}\|^2}{\|T_{\ell m} \wedge T_{\ell'm'}\|^2}
=
\frac{3N^2}{\ell^2(\ell+1)^2} \sum\limits_{\ell' = 2}^{N-1}
\sum\limits_{m' = -\ell'}^{\ell'} \|A_{T_{\ell m}}T_{\ell'm'}\|^2 \\
&\qquad\qquad=
\frac{3N^2}{4\ell^2(\ell+1)^2\hbar_N^2}\left(
\Big\|C_{\ell,m,\ell,1-m}^{1, -1}T_{1, -1}\Big\|^2 +
\Big\|C_{\ell,m,\ell,1-m}^{1, 0}T_{1, 0}\Big\|^2 +
\Big\|C_{\ell,m,\ell,1-m}^{1, 1}T_{1, 1}\Big\|^2\right) \\
&\qquad\qquad= 
\frac{3N}{2\ell^2(\ell+1)^2\hbar_N^2}\left(
\left(C_{\ell,m,\ell,1-m}^{1, -1}\right)^2 +
\left(C_{\ell,m,\ell,1-m}^{1, 0}\right)^2 +
\left(C_{\ell,m,\ell,1-m}^{1, 1}\right)^2\right) \\
&\qquad\qquad= \frac{3N}{2\ell^2(\ell+1)^2}\frac{N^2-1}{4}\frac{12\ell(\ell+1)}{N(N^2-1)} \\
&\qquad\qquad= \frac{9}{2\ell(\ell+1)},
\end{split}
\end{gather*}
by Lemma \ref{lemma_3_sum}.
\end{proof}
The last ingredient we need is the following sum of sectional curvatures of $SU(N)$.
\begin{lemma}\label{lemma_kg}
For any $2 \leq \ell \leq N-1$ and $-\ell \leq m \leq \ell$, we have
\begin{equation}
K_G(T_{\ell m},T_{10}) + K_G(T_{\ell m}, T_{1, -1}) + K_G(T_{\ell m}, T_{1, 1}) = \frac{3}{2\ell(\ell+1)}.
\end{equation}
\end{lemma}

\begin{proof}
Recall Arnold's formula for sectional curvature of a Lie group in the form
\eqref{arnoldsectional}. 
We have $\mathrm{ad}_uv = -\tfrac{1}{\hbar_N}[u, v]$ and $\adstar_uv = \tfrac{1}{\hbar_N}\Delta^{-1}\big[u,\,\Delta v\big]$, where $\Delta$ is the quantized Laplacian. For $u = T_{1m'}$, $-1 \leq m' \leq 1$, and $v = T_{\ell m}$, then in the Zeitlin metric we have:
\begin{gather}\label{eq_ad_adstar}
\begin{split}
 \langle u, v \rangle &= 0; \quad \|u\|^2 = \alpha_1 = 2/N; \quad \|v\|^2 = \alpha_{\ell} = \ell(\ell+1)/N, \\
\adstar_uu &= \adstar_vv = 0, \\ 
\mathrm{ad}_uv &= \mathrm{ad}_{T_{1m'}}T_{\ell m} = \tfrac{1}{\hbar_N}\big[T_{\ell m}, \, T_{1m'}\big] = \tfrac{1}{\hbar_N} C_{\ell m 1 m'}^{\ell, -m}\, T_{\ell, -m}, \\
\adstar_uv &= \adstar_{T_{1m'}}T_{\ell m} = \tfrac{1}{\hbar_N}\Delta^{-1}[T_{1m'},\, \Delta T_{\ell m}] = \tfrac{1}{\hbar_N}C_{1 m' \ell m}^{\ell, -m} \, T_{\ell, -m} = -\tfrac{1}{\hbar_N}C_{\ell m 1 m'}^{\ell, -m}\, T_{\ell, -m}, \\
\adstar_vu &= \adstar_{T_{\ell m}}T_{1m'} = \tfrac{1}{\hbar_N}\Delta^{-1}[T_{\ell m},\, \Delta T_{1m'}] = \frac{2}{\ell(\ell+1)\hbar_N} C_{\ell m 1 m'}^{\ell, -m} \, T_{\ell, -m}.
\end{split}
\end{gather}
It follows from the third and fourth lines in \eqref{eq_ad_adstar} that
\begin{equation*}
\adstar_uv + \mathrm{ad}_uv = \adstar_{T_{1m'}}T_{\ell m} + \mathrm{ad}_{T_{1m'}}T_{\ell m} = 0.
\end{equation*}
Thus, Arnold's formula reduces to
\begin{gather}
\begin{split}
K_G(T_{1m'},T_{\ell m}) &= \frac{1}{4} \frac{\| \adstar_{T_{\ell m}}T_{1m'} \|^2}{\| T_{1m'}\|^2 \| {T_{\ell m}}\|^2-\langle T_{1m'},{T_{\ell m}}\rangle^2} \\ 
&= \frac{N}{4}\frac{\| \adstar_{T_{\ell m}}T_{1m'} \|^2}{2\ell(\ell+1)} \\ 
&= \frac{N}{8\ell(\ell+1)}\left(
\frac{2}{\ell(\ell+1)\hbar_N} C_{\ell m 1 m'}^{\ell, -m}\right)^2 \ell(\ell+1) \\
&= \frac{N}{2\ell^2(\ell+1)^2\hbar_N^2} \left(C_{\ell m 1 m'}^{\ell, -m}\right)^2.
\end{split}
\end{gather}
Adding these up and using Lemma \ref{lemma_3_sum} gives
\begin{gather*}
\begin{split}
\sum\limits_{m' = -1}^1 K_G(T_{1m'}, T_{\ell m})
&= \frac{N}{2\ell^2(\ell+1)^2\hbar_N^2}
\sum\limits_{m' = -1}^1
\left(C_{\ell m 1 m'}^{\ell, -m}\right)^2 \\
&=
\frac{N}{2\ell^2(\ell+1)^2\hbar_N^2}\frac{12\ell(\ell+1)}{N(N^2-1)} \\
&= \frac{3}{2\ell(\ell+1)}.
\end{split}
\end{gather*}
\end{proof}

\begin{proof}[Proof of Theorem \ref{homogeneous_space_thm}] 
Let $\Ric_B$ denote the Ricci curvature of the quotient space $SU(N)/SO(3)$. Summing over $(\ell', m') \neq (\ell, m)$ below, we get 
\begin{gather*}
\begin{split}
\Ric_B(T_{\ell m}, T_{\ell m}) &= \|T_{\ell m}\|^2\sum\limits_{\ell' = 2}^n~\sum\limits_{m' = -\ell'}^{\ell'} K_B(T_{\ell m}, T_{\ell' m'}) \\
&= \frac{\lambda_{\ell}}{N}\sum\limits_{\ell' = 2}^n~\sum\limits_{m' = -\ell'}^{\ell'} \Big(K_G(T_{\ell m}, T_{\ell' m'}) + 3\frac{\|A_{T_{\ell m}}T_{\ell' m'}\|^2}{\|T_{\ell m} \wedge T_{\ell'm'}\|^2}\Big) \\
&= \frac{\lambda_{\ell}}{N}\left( \sum\limits_{\ell' = 2}^n~\sum\limits_{m' = -\ell'}^{\ell'} K_G(T_{\ell m}, T_{\ell' m'})\right) + \frac{9\lambda_{\ell}}{2N\ell(\ell+1)} \qquad\qquad (\text{Lemma \ref{lemma_A_XY}}) \\
&= \Ric(T_{\ell m}, T_{\ell m}) -\!\frac{\lambda_{\ell}}{N}\sum_{m' = -1}^1 K_G(T_{\ell m}, T_{1, m'}) + \frac{9}{2N} \\
&= \Ric(T_{\ell m}, T_{\ell m}) -\frac{3}{2N} + \frac{9}{2N} \qquad\qquad\quad (\text{Lemma \ref{lemma_kg}}) \\
&= \Ric(T_{\ell m}, T_{\ell m}) + \frac{3}{N}.
\end{split}
\end{gather*}
Now, since $\Ric(T_{\ell m}, T_{\ell m}) = r_{\ell}(N)\langle T_{\ell m}, T_{\ell m}\rangle$ by Theorem \ref{main_thm}, we get
\begin{gather*}
\begin{split}
\Ric_B(T_{\ell m}, T_{\ell m}) = r_{\ell}(N)\langle T_{\ell m}, T_{\ell m} \rangle + \frac{3}{N} = \bigg(r_{\ell}(N) + \frac{3}{\lambda_{\ell}}\bigg)\langle T_{\ell m}, T_{\ell m} \rangle.
\end{split}
\end{gather*}

\end{proof}

\bibliographystyle{amsplain}
\bibliography{bibliography}

\end{document}